\documentclass{article}

\usepackage{amsmath}
\usepackage{amsfonts}
\usepackage{amssymb}
\usepackage{mathrsfs}
\usepackage{graphicx}
\usepackage{fourier}
\usepackage{stmaryrd}
\usepackage{bbm}
\usepackage{amsthm}
\usepackage{listings}
\usepackage{enumitem}
\usepackage{caption}
\usepackage{wasysym}

\usepackage{comment}

\usepackage{multirow}

\newtheorem{theorem}{Theorem}[section]
\newtheorem{lem}[theorem]{Lemma}
\newtheorem{prop}[theorem]{Proposition}
\newtheorem{coro}[theorem]{Corollary}

\newcounter{NoMain}

\newtheorem{mainthm}[NoMain]{Theorem}

\theoremstyle{definition}
\newtheorem{definition}[theorem]{Definition}

\theoremstyle{remark}
\newtheorem{remark}[theorem]{Remark}

\usepackage[bookmarks=true]{hyperref}

\numberwithin{equation}{section}

\newcommand{\R}{\mathbb{R}}
\newcommand{\N}{\mathbb{N}}
\newcommand{\Z}{\mathbb{Z}}

\newcommand{\C}{\mathbb{C}}
\newcommand{\U}{\mathcal{U}}
\newcommand{\K}{\mathbb{K}}

\newcommand{\fonction}[5]{#1: \begin{array}{ccc}
#2 & \rightarrow & #3 \\
 #4 & \longmapsto & #5 \end{array}}

 \newcommand{\fon}[4]{\begin{array}{ccc}
#1 & \rightarrow & #2 \\
#3 & \longmapsto & #4 \end{array}}

\DeclareMathOperator{\Ad}{Ad}
\DeclareMathOperator{\vspan}{span}
\DeclareMathOperator{\Id}{Id}
\DeclareMathOperator{\rank}{rank}

\DeclareMathOperator{\Diag}{Diag}

\begin{document}
\title{On quasi-homomorphism rigidity for lattices in simple algebraic groups}
\author{Guillaume Dumas\footnote{Université Claude Bernard Lyon 1, ICJ UMR5208, CNRS, Ecole Centrale de Lyon, INSA Lyon, Université Jean Monnet,
69622 Villeurbanne, France. \href{mailto:gdumas@math.univ-lyon1.fr}{gdumas@math.univ-lyon1.fr}}}

\maketitle
\begin{abstract}
    Property $(TTT)$ was introduced by Ozawa as a strengthening of Kazhdan's property $(T)$ and Burger and Monod's property $(TT)$. In this paper, we improve Ozawa's result by showing that any simple algebraic group of rank $\geq 2$ over a local field has property $(TTT)$. We also show that lattices in a second countable locally compact group inherits property $(TTT)$. Finally, we study to what extent Lie groups with infinite center fail to have properties $(TT)$ and $(TTT)$.
\end{abstract}

\section{Introduction}
\begin{definition}
    Let $G$ be a locally compact group and $H$ an Hilbert space. We say that a Borel locally bounded (i.e. bounded on compact subsets) map $b:G\mapsto H$ with a Borel map $\pi:G\mapsto \mathcal{U}\left(H\right)$ is a\begin{itemize}
        \item cocycle if $\pi$ is a representation and $\forall g,h\in G, b(gh)=b(g)+\pi(g)b(h)$;
        \item quasi-cocycle if $\pi$ is a representation and $\underset{g,h\in G}{\sup} \Vert b(gh)-b(g)-\pi(g)b(h)\Vert <+\infty$;
        \item wq-cocycle if $\underset{g,h\in G}{\sup} \Vert b(gh)-b(g)-\pi(g)b(h)\Vert <+\infty$.
    \end{itemize}
\end{definition}

It is know that $G$ has property $\left(T\right)$ if and only if every cocycle on $G$ is bounded. In \cite{burgermonod}, Burger and Monod introduced a strengthening of property $\left(T\right)$: $G$ has property $(TT)$ if every quasi-cocycle is bounded. In this article, we study a stronger property introduced by Ozawa (\cite{Ozawa+2011+89+104}).

\begin{definition}
    Let $G$ be a locally compact group, $A$ a subgroup of $G$. The pair $(G,A)$ has relative property $(TTT)$ if any wq-cocyle on $G$ is bounded on $A$.
\end{definition}

If $G$ has property $(TT)$, then all quasimorphisms $G\to \R$, that is to say maps $\varphi:G\to \R$ such that $\{\varphi(gh)(\varphi(g)\varphi(h))^{-1} \vert g,h\in G\}$ is relatively compact, are bounded. Property $(TTT)$ allows to study such questions for quasi-homomorphisms, when the target group is no longer $\R$ (see \cite[Thm. A]{Ozawa+2011+89+104}).\medskip

 Ozawa showed that for any local field $\K$, the group $SL_3(\K)$ has property $(TTT)$ as well as all its lattices (\cite[Thm. B]{Ozawa+2011+89+104}). However, it was not clear to him whether property $(TTT)$ passes to lattices that are not cocompact. We show that this is true.

 \begin{mainthm}\label{thm:latticesTTT}
     Let $G$ be a locally compact second countable group and $\Gamma$ a lattice in $G$. Then $G$ has property $(TTT)$ if and only if $\Gamma$ has property $(TTT)$.
 \end{mainthm}

Our main result is an extension of the result on $SL_n$ to higher rank simple algebraic groups.
\begin{mainthm}\label{thm:mainthm}
    Let $G$ be a connected simple algebraic group over a local field $\K$ with $\rank_\K G \geq 2$, then $G\left(\K\right)$ has property $(TTT)$.
\end{mainthm}

We follow the same idea as the classical proof of property $(T)$ for these group: we reduce the proof to the cases of the classical groups $SL_3$ and $Sp_4$. As said before, it is already known that $SL_3$ has property $(TTT)$. We show that for any local field $\K$, $Sp_4(\K)$ has property $(TTT)$ in Theorem \ref{thm:sp4}.\medskip

Finally, Theorem \ref{thm:mainthm} applies to higher rank simple Lie groups with finite center. But when $G$ has infinite center, it is well-known that $G$ has an unbounded quasi-morphism $\phi:G\to \R$ (see \cite[Prop. 6]{Barge1992}). In particular, $G$ does not have property $(TT)$ nor $(TTT)$. However, we show in Proposition \ref{prop:infinitecenter} that the unbounded wq-cocycles of $G$ are completely controlled by the unbounded wq-cocycles of its center.

\subsection*{Acknowledgements}
I am grateful to Mikael de la Salle for introducing me to this questions and for his involvement and helpful suggestions throughout the process of this work. I am thankful to Narutaka Ozawa who took time to listen and answer my questions about his work, and Yves de Cornulier for his explanations on the non-archimedean setting.

\section{Related properties}
\subsection{Positive definite kernels and completely bounded norm}
Let $G$ be a locally compact second countable group. A function $\theta\in L^\infty(G\times G)$ is a positive definite kernel if for any $\xi\in L^1(G)$, $\int \theta(x,y)\xi(x)\overline{\xi(y)}dxdy\geq 0$. Equivalently, $\theta$ is a positive definite kernel if and only if there exists a separable Hilbert space $H$ and a measurable map $F:G\to H$ such that $\theta(x,y)=\langle F(x),F(y)\rangle$ almost everywhere (see \cite[Appendix D]{brownc}). If $\theta$ is continuous, $F$ can be taken continuous and equality holds everywhere. We say that $\theta$ is normalized if $\theta(x,x)=1$ for any $x\in G$. In that case, there is an inequality that will be useful throughout the paper. Let $x,y,z\in G$, we have \begin{align*}
   \vert \theta(x,z)-\theta(y,z)\vert  & =  \vert\langle F(x)-F(y),F(z)\rangle\vert  \\
     & \leq  \Vert F(x)-F(y)\Vert\\
     & \leq  \left(\Vert F(x)\Vert^2+\Vert F(y)\Vert^2 - \langle F(x),F(y)\rangle- \langle F(y),F(x)\rangle\right)^{1/2}\\
     & \leq  \left( 2-\theta(x,y)-\overline{\theta(x,y)}\right)^{1/2}\\
     & \leq  \sqrt{2}\left\vert \theta(x,y)-1\right\vert^{1/2}.
\end{align*}

Let $\theta\in L^\infty(G\times G)$, we define the cb-norm of $\theta$ by $$\Vert \theta\Vert_{cb}=\inf \left\lbrace\Vert P\Vert \Vert Q\Vert : P,Q\in L^\infty(G;H), \theta(x,y)=\langle P(x),Q(y)\rangle\right\rbrace.$$

\subsection{Property \texorpdfstring{$(T_P)$}{(Tp)} and \texorpdfstring{$(T_Q)$}{(Tq)}}
Let $G$ be a locally compact second countable group and $A$ a subgroup of $G$.
\begin{definition}[\cite{Ozawa+2011+89+104}]
    The pair $(G,A)$ has relative property $(T_P)$ if $\forall \varepsilon>0$, $\exists\delta >0$ and $K\subset G$ compact such that for any $\theta:G\times G\mapsto \C$ Borel normalized positive definite kernel verifying \begin{equation}
        \label{eq:tp_h1} \underset{g\in G}{\sup} \left\Vert \theta(g\cdot,g\cdot)-\theta \right\Vert_{cb} < \delta
    \end{equation}and \begin{equation}
        \label{eq:tp_h2}\underset{g^{-1}h\in K}{\sup} \left\vert \theta(g,h)-1\right\vert < \delta
    \end{equation} 
    then \begin{equation}
        \label{eq:tp_c} \underset{x,y\in A}{\sup} \left\vert \theta(x,y)-1\right\vert < \varepsilon.
    \end{equation}
\end{definition}

\begin{remark}As noticed by Ozawa (\cite[Section 3]{Ozawa+2011+89+104}), it is enough to consider only continuous kernels instead of Borel kernels. Furthermore, the hypothesis \eqref{eq:tp_h2} can be weakened to \begin{equation}
        \label{eq:tp_hweak}
    \underset{x\in K}{\sup}\left\vert \theta(x,1)-1\right\vert < \delta.\end{equation}
    Indeed, if $\theta$ verifies \eqref{eq:tp_h1} and \eqref{eq:tp_hweak}, then for any $g,h\in G$ with $g^{-1}h\in K$, there is $x\in K$ such that $h=gx$ so \begin{align*}
        \left\vert \theta(g,h)-1\right\vert & =  \left\vert \theta(g,gx)-1\right\vert \\
         & \leq   \left\vert \theta(g,gx)-\theta(1,x)\right\vert +\left\vert \theta(x,1)-1\right\vert\\
         & \leq  2\delta.
    \end{align*}
\end{remark}

\begin{definition}
    The pair $(G,A)$ has relative property $(T_Q)$ if $\forall \varepsilon>0$, $\exists\delta >0$ and $K\subset G$ compact such that for any Borel map $\pi:G\mapsto \U(H)$ and every unit vector $\xi\in H$ verifying \begin{equation}
        \label{eq:tq_h1} \underset{g,h\in G}{\sup} \left\Vert \pi(gh)\xi-\pi(g)\pi(h)\xi\right\Vert<\delta
    \end{equation}and \begin{equation}
        \label{eq:tq_h2} \underset{g\in K}{\sup} \left\Vert \pi(g)\xi-\xi\right\Vert < \delta
    \end{equation}
    then \begin{equation}
        \label{eq:tq_c} \underset{x\in A}{\sup} \left\Vert \pi(x)\xi-\xi\right\Vert < \varepsilon.
    \end{equation}
\end{definition}

\subsection{Measurable factorisation}
Let $X$ be a $\sigma$-finite measure space such that $L^2\left(X\right)$ is a separable Hilbert space, for example $X$ a locally compact second countable group. Then $L^1(X)$ is also separable. If $E$ is a separable Banach space, a function $\phi:X\to E^*$ is $w^*$-measurable if $x\mapsto \langle \phi(x),v\rangle$ is measurable for any $v\in E$. Since $E$ is separable, let $(x_n)$ be a dense sequence in the unit sphere of $E$. Then $\Vert \phi(\cdot)\Vert= \underset{n\in \N}{\sup} \vert \varphi(\cdot)(x_n)\vert$ is measurable, as the supremum of measurable functions. Thus, we can define $L^p_\sigma(X;E^*)$ as the space of $w^*$-measurable functions $\phi:X\to E^*$ such that $$\Vert \phi \Vert_p=\Vert \Vert \phi(.)\Vert \Vert_p<+\infty$$(see \cite{diestel1977vector} for more details). By Pettis mesurability theorem (\cite[Ch. II, Thm. 2]{diestel1977vector}), if $E^*$ is separable and $\phi:X\to E^*$ is such that $x\mapsto u(\phi(x))$ is measurable for any $u\in E^{**}$, then $\phi$ is Bochner measurable. This implies that when $E$ is a separable reflexive Banach space, the space $L^p_\sigma(X;E^*)$ coincides with the space $L^p(X;E^*)$ of (Bochner) measurable functions. This holds more generally when $E^*$ has the Radon-Nikodym property (see \cite[Ch. IV]{diestel1977vector}).\medskip

Let $E,F$ be two Banach spaces, we denote $E\hat{\otimes}F$ the completion of $E\times F$ for the projective tensor norm (see \cite[Ch. VIII]{diestel1977vector}). When $E,F$ are separable, this is a separable Banach space. By \cite[Ch. VIII.2, Coro. 2]{diestel1977vector}, there is an isometric isomorphisms \begin{equation}
    \label{eq:predualB}
\left(E\hat{\otimes}F\right)^*\simeq B\left(E,F^*\right)\end{equation} and $\phi:E\hat{\otimes}F\to \C$ corresponds to the unique bounded operator $u:E\to F^*$ such that $\forall x,y\in E\times F$, $\phi(x\otimes y)=u(x)(y)$. Thus, we can define the spaces $L^\infty_\sigma\left(X;B\left(E,F^*\right)\right)$ when $E,F$ are separable Banach spaces.\medskip

Let $E$ be a Banach space. By \cite[Ch. VIII.1, Ex.10]{diestel1977vector}, the natural embedding $L^1(X)\otimes E\to L^1(X;E)$ extends to an isometric isomorphism \begin{equation}
    \label{eq:bochnertensor} L^1(X)\hat{\otimes} E \simeq L^1(X;E).
\end{equation}

Furthermore, if $E$ is separable, the map \begin{equation}
    \label{eq:dualL1} \fon{L^\infty_\sigma(X;E^*)}{L^1(X;E)^*}{\xi}{u\mapsto \int_X  \left[\xi(x)\right](u(x)) dx}
\end{equation}is an isometric isomorphism (see \cite[Thm. 1.16]{coine-these} or \cite[Prop. 2.20, 2.26 and Thm. 2.29]{Pisier_2016}).\medskip

Let $H$ be a separable Hilbert space. Combining \eqref{eq:bochnertensor} and \eqref{eq:dualL1}, a function in $L^\infty(X;H)=L^\infty_\sigma(X;H^*)$ corresponds to a functional $\phi$ on $L^1(X)\hat{\otimes} H$, which is defined on simple tensors $u\otimes y\in L^1(X)\otimes H$ by $$\phi(u\otimes y)=\int_X u(x)\langle \xi(x),y\rangle dx.$$ Thus by \eqref{eq:predualB}, the map \begin{equation}
    \label{eq:isometric-integration}\fonction{T}{L^\infty(X;H)}{B(L^1(X),H)}{\xi}{u\mapsto \int_X u(x)\xi(x)dx}
\end{equation}is an isometric isomorphism.\medskip

Let $E,F$ be two separable Banach spaces. The above properties give isometric isomorphisms 
\begin{align*}
    L_\sigma^\infty\left(X;B(E,F^*)\right) & \simeq   L_\sigma^\infty\left(X;(E\hat{\otimes}F)^*\right) & \textrm{ by }\eqref{eq:predualB} \\
     & \simeq  L^1\left(X;E\hat{\otimes}F\right)^* & \textrm{ by }\eqref{eq:dualL1} \\
     & \simeq  \left(L^1(X)\hat{\otimes}(E\hat{\otimes}F) \right)^*  & \textrm{ by }\eqref{eq:bochnertensor} \\
     & \simeq \left(E \hat{\otimes} L^1(X) \hat{\otimes} F\right)^* &\\
     & \simeq  \left(E \hat{\otimes} L^1(X;F)\right)^* & \textrm{ by }\eqref{eq:bochnertensor}\\
     & \simeq  B\left(E,L^1(X,F)^*\right) & \textrm{ by }\eqref{eq:predualB} \\
     & \simeq B\left(E,L^\infty_\sigma\left(X;F^*\right)\right) & \textrm{ by }\eqref{eq:dualL1} \\
\end{align*}and following the path of isomorphism shows that \begin{equation}
    \label{eq:evaluationweakmes}\fon{L_\sigma^\infty\left(X;B\left(E,F^*\right)\right)}{B\left(E,L^\infty_\sigma\left(X;F^*\right)\right)}{\xi}{u\mapsto \xi(\cdot)(u)}.
\end{equation}

 Let $$\Gamma_2\left(L^1(X),L^\infty(X)\right)=\left\lbrace T\in B(L^1(X),L^\infty(X)) \left\vert \begin{aligned}[c]
     T=SR \textrm{ where }
     R\in B(L^1(X),H),\\
     S\in B(H,L^\infty(X)) \textrm{ for some }\\\textrm{ separable Hilbert space }H
 \end{aligned}\right.\right\rbrace$$with norm $\gamma(T)=\inf \Vert S\Vert \Vert R\Vert$. Let $z\in L^1(X)\otimes L^1(X)$, we define $$\Vert z \Vert_*=\inf \left( \sum \Vert u_i\Vert^2 \right)^{1/2}\left(\sum  \Vert v_i\Vert^2\right)^{1/2}$$where the infimum runs over all finite families $(u_i),(v_i)$ such that for $\xi,\eta\in (L^1(X))^*$, $$\vert (\xi\otimes \eta)(z)\vert \leq \left( \sum \vert  \xi(u_i)\vert^2 \right)^{1/2}\left(\sum \vert  \eta(v_i)\vert^2\right)^{1/2}.$$Then, $\Vert \cdot \Vert_*$ is a norm on $L^1(X)\otimes L^1(X)$. By \cite[Thm. 2.8 and Coro. 2.9]{pisier1986factorization}, there is an isometric isomorphism \begin{equation}
     \label{eq:predualGamma} \Gamma_2\left(L^1(X),L^\infty(X)\right)\simeq \left(L^1(X)\otimes_* L^1(X)\right)^*
 \end{equation}where $L^1(X)\otimes_* L^1(X)$ is the completion of the tensor product $L^1(X)\otimes L^1(X)$ for the norm $\Vert \cdot \Vert_*$. Thus, this space has a separable predual and we can consider the spaces $L^\infty_\sigma\left(Y;\Gamma_2\left(L^1(X),L^\infty(X)\right)\right)$.\medskip

 If $\varphi\in L^\infty(X\times X)$, we can define $r_\varphi\in B\left(L^1(X),L^\infty(X)\right)$ by $$r_\varphi(f)(s)=\int_X f(t)\varphi(t,s)dt.$$
 By \cite[Thm. 3.3]{spronk}, $\varphi$ is a Schur multiplier if and only if $r_\varphi\in \Gamma_2\left(L^1(X),L^\infty(X)\right)$, and in that case, $\Vert \varphi\Vert_{cb}=\gamma(r_\varphi)$.\medskip

 Let $\phi\in L^\infty(X\times X\times X)$ and denote $\phi_x=\phi(\cdot,x,\cdot)$. Such a map defines an operator $\Tilde{\phi}\in L^\infty_\sigma\left(X;B\left(L^1(X),L^\infty(X)\right)\right)$ by $$\Tilde{\phi}(x)(u)=\int_X \phi(t,x,\cdot)u(t)dt=r_{\phi_x}(u).$$
\begin{prop}\label{prop:factorisation}
    Let $G$ be a locally compact second countable group. Let $\theta\in L^{\infty}(G\times G)$ be a positive definite kernel on $G$ such that for any $g$, $\Vert g\theta-\theta\Vert_{cb}\leq \delta$. Denote $\phi(x,g,y)=\theta(gx,gy)-\theta(x,y)$. Then there exists a separable Hilbert space $H$ and two functions $a,b\in L^\infty_\sigma\left(G;B\left(L^1(G),H\right)\right)$ such that for almost every $g\in G$ and for every $u,v\in L^1(G)$, $$\left[\Tilde{\phi}(g)(u)\right](v)=\langle a(g)(u),b(g)(u)\rangle$$ with $\Vert a\Vert_\infty \Vert b\Vert_\infty\leq \delta$.
\end{prop}

\begin{proof}
    Since $\phi_g$ is a Schur multiplier for any $g\in G$, we have $$\Tilde{\phi}\in L^\infty_\sigma\left(G;\Gamma_2\left(L^1(G),L^\infty(G)\right)\right)$$with $\Vert \Tilde{\phi}\Vert_{\infty,\Gamma_2}=\underset{g\in G}{\sup} \gamma(r_{\psi_{g}})\leq \delta$.
   The result is then a direct consequence of \cite[Thm 5.1]{coine}.\qedhere

\end{proof}

\begin{lem}\label{lem:unitarymeas}
 Let $H$ be a separable Hilbert space, $X,Y$ measured spaces such that $L^2(X),L^2(Y)$ are separable and $Y$ is complete. Let $\alpha,\beta\in L^\infty_\sigma\left(Y;B\left(L^1(X),H\right)\right)$ be two maps such that for almost every $y\in Y$ and every $u,v\in L^1(X)$, \begin{equation}
     \label{eq:hypogeom}\langle \alpha(y)(u),\alpha(y)(v)\rangle=\langle \beta(y)(y),\beta(y)(v)\rangle.
 \end{equation}Then there exists a map $\pi:Y\to \mathcal{U}\left(H\oplus \ell^2(\N)\right)$ which is measurable when the group $\mathcal{U}(H\oplus \ell^2(\N))$ is endowed with the Borel $\sigma$-algebra coming from the strong operator topology, such that for almost all $y\in Y$, for all $u\in L^1(X)$, $U_y\left(\alpha(y)(u)\right)=\beta(y)(u)$.
\end{lem}

\begin{proof}
    First, by \eqref{eq:evaluationweakmes}the map $$\fon{L^\infty_\sigma\left(Y;B\left(L^1(X),H\right)\right)}{B\left(L^1(X),L^\infty_\sigma\left(Y;H\right)\right)}{\alpha}{u\mapsto \alpha(\cdot)(u)}$$is an isometric isomorphism. Furthermore, $L^\infty_\sigma(Y;H)=L^\infty(Y;H)$ since $H$ is a separable Hilbert space. Thus, for $u\in L^1(X)$, the maps $y\mapsto \alpha(y)(u)$ and $y\mapsto \beta(y)(u)$ are measurable.\smallskip

    Set $H'=H\oplus \ell^2(\N)$. Since $L^1(X)$ is separable, we can consider $(u_n)_{n\in \N}$ a dense sequence in $L^1(X)$. Denote $Y'$ a conull set in $Y$ such that $\eqref{eq:hypogeom}$ holds for all $y\in Y'$.\smallskip

    If $y\in Y$, define $H_y=\overline{\alpha(y)(L^1(X))}$, then the sequence $\left(\alpha(y)(u_n)\right)_{n\in \N}$ is dense in $H_y$. We apply the Gram-Schmidt process to this family: set $a_0(y)=\alpha(y)(u_0)$ which is measurable. If we have constructed $a_0(y),\dots,a_{n-1}(y)$ such that $$\vspan(a_0(y),\dots,a_{n-1}(y))=\vspan (\alpha(y)(u_0),\dots,\alpha(y)(u_{n-1}))$$ and each $a_k$ is measurable, we set $$a_n(y)=\alpha(y)(u_n)-\sum_{k<n,a_k(y)\neq 0} \frac{\langle a_k(y), \alpha(y)(u_n)\rangle}{\Vert a_k(y)\Vert^2} a_k(y).$$Recursively, this give a family of vectors $(a_n(y))_{n\in \N}$ which for each $y$ contains an orthogonal basis of $H_y$ and some zero vectors. Since $\{y \vert a_n(y)\neq 0\}$ is measurable, replacing $a_n(y)$ by $a_n(y)/\Vert a_n(y)\Vert$ on this set still gives a measurable function, and now $(a_n(y))$ contains an orthonormal basis and some zero vectors for each $y\in Y$.\smallskip

    With the same process, we construct for each $y\in Y$ a family $\left(b_n(y)\right)_{n\in \N}$ containing an orthonormal basis of $K_y=\overline{\beta(y)(L^1(X))}$ and some zero vectors such that for each $n\in \N$, $y\mapsto b_n(y)$ is measurable.\smallskip

    The crucial point is that using the hypothesis \eqref{eq:hypogeom}, for any $y\in Y'$ we have \begin{equation}
        \label{eq:null} a_n(y)=0 \Longleftrightarrow b_n(y)=0
    \end{equation}
    and
    \begin{equation}
        \label{eq:decompo} a_n(y)=\sum_{k=0}^n \lambda_k(y)\alpha(y)(u_k) \Longleftrightarrow b_n(y)= \sum_{k=0}^n \lambda_k(y)\beta(y)(u_k).
    \end{equation}

Now, consider an orthonormal basis $(e_n)_{n\in \N}$ of $H$ and an orthonormal basis $(f_n)_{n\in \N}$ of $\ell^2(\N)$. Since $\ell^2(\N)$ has uncountable dimension, there exists $(f'_n)_{n\in \N}$ such that $(f_n)\cup (f'_n)$ is linearly independent. Let $(g_n)=(e_n+f'_n)\cup (f_n)$. This is a total family in $H'=H\oplus \ell^2(\N)$. Let $$c_n(y)=P_{H_y^\bot}(g_n)=g_n-\sum \langle a_n(y),g_n \rangle a_n(y)$$and $$d_n(y)=P_{K_y^\bot}(g_n)=g_n-\sum \langle b_n(y),g_n \rangle b_n(y).$$As limits of measurable functions, $c_n,d_n$ are measurable since $Y$ is complete. The family $(c_n(y))_{n\in \N}$ is total in $H_y^\bot$, and linearly independent. Indeed, if there is a relation $\sum_{i=1}^n \lambda_i c_i(y)=0$, then $\sum_{i=1}^n \lambda_ig_i\in H_y$, but $\left(\vspan (g_n)_{n\in\N}\right)\cap H=\{0\}$ by construction, so $\lambda_i=0$ for any $1\leq i\leq n$.

Similarly, the family $\left(d_n(y)\right)$ is total in $K_y^\bot$ and linearly independent. Thus, applying the Gram-Schmidt process produces $\left(a'_n(y)\right)_{n\in\N}$ and $\left(b'_n(y)\right)_{n\in\N}$, which are also measurable functions and an orthonormal basis of $H_y^\bot,K_y^\bot$ respectively.\medskip

For $y\in Y'$, we have two orthonormal bases of $H'=H\oplus \ell^2(\N)$. Thus, there is a unique unitary map $U_y$ sending $a_n(y)$ to $b_n(y)$ and $a'_n(y)$ to $b'_n(y)$, using \eqref{eq:null} to ensure that $U_y$ is well-defined on the zero vectors in $\left(a_n(y)\right)_{n\in \N}$. On $Y\setminus Y'$, we set $U_y=\Id$.\smallskip

Using $\eqref{eq:decompo}$ we show recursively that for any $y\in Y'$, $n\in \N$, $$U_y(\alpha(y)(u_n))=\beta(y)(u_n).$$ Thus, by density of $(u_n)$ and continuity of $U_y$, we get that for any $u\in L^1(X)$, $$U_y(\alpha(y)(u))=\beta(y)(u).$$

Let $\xi\in H'$, then for $y\in Y'$, $$\displaystyle \xi=\sum_{n\geq 0} \left(\langle a_n(y),\xi \rangle a_n(y) + \langle a'_n(y),\xi \rangle a'_n(y)\right)$$so $$\displaystyle U_y\xi=\sum_{n\geq 0} \left(\langle a_n(y),\xi \rangle b_n(y) + \langle a'_n(y),\xi \rangle b'_n(y)\right).$$ Again since $Y$ is complete, $y\mapsto U_y\xi$ is measurable as a pointwise limit almost everywhere of measurable functions.\\
Since this is true for any $\xi\in H'$ and since $H'$ is separable, this implies that $y\mapsto U_y$ is measurable for the strong operator topology on $\mathcal{U}\left(H'\right)$.
\end{proof}

\subsection{Relation between properties}
Ozawa showed the following implications between these strengthenings of property $(T)$ (\cite[Thm. 1]{Ozawa+2011+89+104}).
\begin{theorem}\label{thm:ozawa_impl}
    $$\textrm{rel. property }(T_P) \Longrightarrow \textrm{rel. property }(TTT)\Longrightarrow \textrm{rel. property }(T_Q).  $$
\end{theorem}

We aim to show that these properties are all equivalent.

\begin{theorem}\label{thm:equiv}
    If $G$ is a second countable locally compact group and $A$ a subgroup of $G$, then if $(G,A)$ has relative property $(T_Q)$, it has relative property $(T_P)$.
\end{theorem}
\begin{proof}
    Let $\varepsilon>0$ and $\theta$ be a continuous positive definite normalized kernel verifying \eqref{eq:tp_h1} and \eqref{eq:tp_h2} for some $\delta,K$ to be determined later.
By definition, there exists a separable Hilbert space $H$ and a continuous map $\xi:G\to H$ such that $\forall g,h\in G$, $\theta(g,h)=\langle \xi(g),\xi(h) \rangle$ and $\forall g\in G$, $\Vert \xi(g)\Vert=1$.\medskip

By Proposition \ref{prop:factorisation}, there exists a separable Hilbert space $H'$ and two functions $a,b\in L^\infty_\sigma(G;B(L^1(G),H'))$ such that for almost every $g\in G$ and for all $u,v\in L^1(G)$, $$\int_{G\times G} \left(\theta(gx,gy)-\theta(x,y)\right) u(x)v(y)dxdy = \langle a(g)(u),b(g)(v)\rangle$$ with $\Vert a\Vert_\infty\Vert b\Vert_\infty <\delta$. Up to multiplying $a,b$ by some constant, we can actually assume that $\Vert a\Vert_\infty<\sqrt{\delta}$ and $\Vert b\Vert_\infty<\sqrt{\delta}$. 

With the notation of \eqref{eq:isometric-integration}, we also get \begin{multline*}
\int_{G\times G} \left(\theta(gx,gy)-\theta(x,y)\right) u(x)v(y)dxdy \\
    = \langle T(g^{-1}\xi)(u),T(g^{-1}\xi)(v)\rangle - \langle T(\xi)(u),T(\xi)(v)\rangle.
\end{multline*}
But then, setting $\Tilde{a}(g)(u)=\frac{a(g)(u)+b(g)(u)}{2}$ and $\Tilde{b}(g)(u)=\frac{a(g)(u)-b(g)(u)}{2}$, we also have $\Vert \Tilde{a}\Vert_\infty<\sqrt{\delta}$ and $\Vert \Tilde{b}\Vert_\infty<\sqrt{\delta}$. In the space $H\oplus H'$, we have for any $u,v\in L^1(G)$ and almost every $g\in G$, we get \begin{multline*}\langle (T(\xi)(u),\Tilde{a}(g)(u)),(T(\xi)(v),\Tilde{a}(g)(v))\rangle \\=\langle (T(g^{-1}\xi)(u),\Tilde{b}(g)(u)),(T(g^{-1}\xi)(v),\Tilde{b}(g)(v))\rangle.\end{multline*}

We apply Lemma \ref{lem:unitarymeas} to $X=Y=G$ and $$\alpha(g)(u)=(T(\xi)(u),\Tilde{a}(g)(u)), \beta(g)(u)=(T(g^{-1}\xi)(u),\Tilde{b}(g)(u)),$$ to get a map $\pi:G\to \mathcal{U}(H\oplus H'\oplus \ell^2(\N))$ which is measurable for the completion of the Borel $\sigma$-algebra on $G$, and such that for almost every $g\in G$ and every $u\in L^1(G)$, \begin{equation}
    \label{eq:pialmostrep}\pi(g)(T(\xi)(u),\Tilde{a}(g)(u))=(T(g^{-1}\xi)(u),\Tilde{b}(g)(u)).
\end{equation}Then, \begin{align*}
    \Vert T(g^{-1}\xi)(u) - \pi(g)T(\xi)(u) \Vert & \leq  \Vert \Tilde{b}(g)(u)\Vert+\Vert(T(g^{-1}\xi)(u),\Tilde{b}(g)(u))-\pi(g)T(\xi)(u) \Vert  \\
     & \leq \sqrt{\delta}\Vert u\Vert + \Vert \pi(g)(T(\xi)(u),\Tilde{a}(g)(u))-\pi(g)T(\xi)(u)\Vert\\
     & \leq  \sqrt{\delta}\Vert u\Vert+\Vert \Tilde{a}(g)(u)\Vert\\
     & \leq 2\sqrt{\delta}\Vert u\Vert.
\end{align*}

But since $$\pi(g)T(\xi)(u)=\pi(g)\int_G u(x)\xi(x)dx=\int_G u(g)\pi(g)(\xi(x))dx=T(\pi(g)\circ \xi)(u),$$we get that for almost every $g$, $$\Vert T(g^{-1}\xi-\pi(g)\circ \xi)(u)\Vert \leq 2\sqrt{\delta}\Vert u\Vert$$
thus $$\Vert T(g^{-1}\xi-\pi(g)\circ \xi)\Vert_{B(L^1(X),H)} \leq 2\sqrt{\delta}.$$Since $T$ is an isometry and $\xi$ is continuous, for almost every $g\in G$ and for all $x\in G$, \begin{equation}
    \label{eq:star} \Vert \xi(gx)-\pi(g)\xi(x)\Vert \leq 2\sqrt{\delta}. 
\end{equation}

We want to change $\pi$ so that $\eqref{eq:pialmostrep}$ holds everywhere and $\pi$ is a Borel map. We proceed as in \cite{Ozawa+2011+89+104}. Let $M$ be a Borel subset of $G$ of measure zero such that $\eqref{eq:pialmostrep}$ holds for all $g\in G\setminus M$. There exists also a Borel subset $N$ of measure zero such that $\pi$ is Borel $G\setminus N$. By regularity of the Haar measure, there exists a $G_\delta$ set of measure zero $N'=\bigcap_n U_n$ such that $M\cup N\subset N'$. Let $K$ be any compact neighborhood of $G$ and consider the map multiplication map $m:(G\setminus N')\times (K\setminus N')\to G$. Since $N'$ has zero measure and $K$ positive measure, $m$ is surjective. Furthermore, for any $g\in G$, \begin{align*}
   m^{-1}(\{g\})  & = \left\lbrace (gk^{-1},k) \vert k\in K\setminus N', xk^{-1}\in G\setminus N'\right\rbrace \\
     & =  \bigcup_{p,q\in \N} \left\lbrace(gk^{-1},k) \vert k\in K\cap U_p^\mathsf{c}\right\rbrace \cap \left((G\setminus U_n)\times K\right)
\end{align*}so $m^{-1}(\{g\})$ is $\sigma$-compact. Thus, applying the Lusin-Novikov uniformization theorem (\cite[Thm.35.46]{kechris2012classical}), there exists a Borel section $s:G\to (G\setminus N')\times (K\setminus N')$ of $m$. Then $t=p_K\circ s:G\to K$ is a Borel map  such that $\forall g\in G$, $gt_g^{-1},t_g\in G\setminus N'$.\\
Set $\Tilde{\pi}(g)=\pi(gt_g^{-1})\pi(t_g)$, this is a Borel map and $\forall g\in G$, \begin{align*}
   \Vert \xi(gx)-\Tilde{\pi}(g)\xi(x)\Vert  & \leq  \Vert \xi(gx)-\pi(gt_g^{-1})\xi(t_gx)\Vert\\&\phantom{{}\leq}+\Vert \pi(gt_g^{-1})\xi(t_gx)- \pi(gt_g^{-1})\pi(t_g)\xi(x)\Vert \\
     & \leq  2\sqrt{\delta}+\Vert  \xi(t_gx)- \pi(t_g)\xi(x)\Vert \\
     & \leq  4\sqrt{\delta}
\end{align*}since $\eqref{eq:star}$ holds for $t_g$ and $gt_g^{-1}$.\medskip

Let $\xi=\xi(e)$. Let us show that the pair $(\Tilde{\pi},\xi)$ verifies \eqref{eq:tq_h1} and \eqref{eq:tq_h2} to apply relative property $(T_Q)$.

By hypothesis \eqref{eq:tp_h2}, we have for any $g\in G,x\in K$, $$\vert \theta(g,gx)-1\vert < \delta
  \Longleftrightarrow    \vert \langle \xi(g),\xi(gx)\rangle -1\vert < \delta.$$
Thus, \begin{align*}
    \Vert \xi(gx)-\xi(g)\Vert^2 & =  \left\vert\Vert \xi(gx)\Vert^2+\Vert \xi(g)\Vert^2 - \langle \xi(gx),\xi(g)\rangle - \langle \xi(g),\xi(gx)\rangle \right\vert\\
     & \leq  \vert 1-\langle \xi(gx),\xi(g)\rangle \vert + \vert 1-\langle \xi(g),\xi(gx)\rangle\vert\\
     & \leq  2\vert \langle \xi(g),\xi(gx)\rangle -1\vert\\
     & \leq  2\delta.
\end{align*}Hence, for any $g\in G,x\in K$, we have \begin{equation} \label{eq:moveinK} \Vert \xi(gx)-\xi(g)\Vert < \sqrt{2\delta}.
\end{equation}

Then, if $x\in K$, we have $$\Vert\Tilde{\pi}(x)\xi-\xi\Vert\leq \Vert \Tilde{\pi}(x)\xi(e)-\xi(x)\Vert+\Vert \xi(x)-\xi(e)\Vert \leq (4+\sqrt{2})\sqrt{\delta}=\delta'$$ by \eqref{eq:star} and \eqref{eq:moveinK}.\medskip

Let $g,h\in G$. We have that \begin{align*}
    \Vert \Tilde{\pi}(gh)\xi-\Tilde{\pi}(g)\Tilde{\pi}(h)\xi\Vert & \leq  \Vert \Tilde{\pi}(gh)\xi(e)-\xi(gh)\Vert+\Vert \xi(gh)-\Tilde{\pi}(g)\xi(h)\Vert\\
    &\phantom{{}\leq}+\Vert \Tilde{\pi}(g)\xi(h)-\Tilde{\pi}(g)\Tilde{\pi}(h)\xi(e)\Vert\\ 
     & <  4\sqrt{\delta}+4\sqrt{\delta}+\Vert \xi(h)-\Tilde{\pi}(h)\xi(e)\Vert\\
     & <  12\sqrt{\delta}=\delta''.
\end{align*}

Now since $(G,A)$ has relative property $(T_Q)$, choosing $K$ associated to $\varepsilon$ in $(T_Q)$ and $\delta$ small enough so that $\delta',\delta''$ are associated to $\varepsilon$, $(\Tilde{\pi},\xi)$ verifies \eqref{eq:tq_h1} and \eqref{eq:tq_h2}. Then we have by relative $(T_Q)$ \eqref{eq:tq_c} that for any $x\in A$, $$\Vert \Tilde{\pi}(x)\xi-\xi\Vert=\Vert \Tilde{\pi}(x)\xi(e)-\xi(e)\Vert<\varepsilon.$$ 

Let $x,y\in A$, \begin{align*}
   \vert \theta(x,y)-1\vert^2  & =  \vert \langle \xi(x),\xi(y)\rangle -1\vert^2  \\
     & =  (1-\langle \xi(x),\xi(y)\rangle)(1-\overline{\langle \xi(x),\xi(y)\rangle})\\
     & =  1-\langle \xi(x),\xi(y)\rangle-\overline{\langle \xi(x),\xi(y)\rangle}+ \vert \langle \xi(x),\xi(y)\rangle\vert^2\\
     & \leq  2-\langle \xi(x),\xi(y)\rangle-\overline{\langle \xi(x),\xi(y)\rangle} \\
     & =  \Vert \xi(x)-\xi(y)\Vert^2
\end{align*}
thus \begin{align*}
    \vert \theta(x,y)-1\vert & \leq  \Vert \xi(x)-\xi(y)\Vert\\
     & \leq  \Vert \xi(x)-\Tilde{\pi}(x)\xi(e)\Vert + \Vert\Tilde{\pi}(x)\xi(e)-\xi(e)\Vert\\
     &\phantom{{}\leq}+\Vert \xi(e)-\Tilde{\pi}(y)\xi(e)\Vert+\Vert \Tilde{\pi}(y)\xi(e)-\xi(y)\Vert\\
     & \leq  4\sqrt{\delta}+\varepsilon+\varepsilon+4\sqrt{\delta}=\varepsilon'
\end{align*} by relative $(T_Q)$ and by \eqref{eq:star}.\medskip

Hence, we showed \eqref{eq:tp_c} for $\varepsilon'$, so $(G,A)$ has relative property $(T_P)$.
\end{proof}

It was shown in \cite{Ozawa+2011+89+104} that both $(T_P)$ and $(T_Q)$ passes to lattices, but as noticed in the introduction, it was not clear whether $(TTT)$ passes to non cocompact lattices. The equivalence of these three properties immediately implies Theorem \ref{thm:latticesTTT}.

\begin{coro}
    Let $G$ be a locally compact group and $\Gamma$ a lattice in $G$, then $G$ has $(TTT)$ if and only if $\Gamma$ has $(TTT)$.
\end{coro}

\section{The symplectic group \texorpdfstring{$Sp_4(\K)$}{Sp4(K)}}
Let $\K$ be a local field. We consider the symplectic group $$Sp_4(\K)=\left\lbrace g\in GL_4(\K) \vert {}^tgJg=J\right\rbrace$$ where $J=\begin{pmatrix}
    0 & I_2 \\ -I_2 & 0
\end{pmatrix}$. Let also $S^{2*}(\K^2)$ be the vector space of symmetric bilinear form on $\K^2$ which can be identified with the space of symmetric matrices in $M_2(\K)$. Then the group $SL_2(\R)$ acts on $S^{2*}(\K^2)$ by $g.B=gB{}^tg$.\medskip

Consider the subgroup $$G_2=\left\lbrace g_A=\begin{pmatrix}A & 0\\0& {}^tA^{-1} \end{pmatrix} \mid A\in SL_2(\K)\right\rbrace\simeq SL_2(\K)$$as well as the two subgroups$$N_2^{+}=\left\lbrace X_B^+=\begin{pmatrix}I_2 & B\\0&I_2\end{pmatrix}\mid B\in M_2(\K), {}^tB=B 
\right\rbrace$$and$$N_2^-=\left\lbrace X_B^-=\begin{pmatrix}I_2 & 0\\B&I_2 \end{pmatrix}\mid B\in M_2(\K), {}^tB=B \right\rbrace.$$

Then the maps $$\fonction{\iota_1}{SL_2(\K)\ltimes S^{2*}(\K^2)}{Sp_4(\K)}{(A,B)}{X_B^+g_A}$$and $$\fonction{\iota_2}{SL_2(\K)\ltimes S^{2*}(\K^2)}{Sp_4(\K)}{(A,B)}{X_B^-g_{{}^tA^{-1}}}$$define two group embeddings of $SL_2(\K)\ltimes S^{2*}(\K^2)$ with $N_2^{+},N_2^{-}$ as images of $S^{2*}(\K^2)$.\medskip

It is known that the pair $(SL_2(\K)\ltimes S^{2*}(\K^2),S^{2*}(\K^2))$ has relative property $(T)$ (see \cite[Coro. 1.5.2]{bekka_de_la_harpe_valette_2008}) thus by \cite[Prop. 3]{Ozawa+2011+89+104}, it has relative property $(T_P)$.\medskip

\begin{theorem}\label{thm:sp4}
Let $\K$ be a local field, the group $Sp(4,\K)$ has property $(T_P)$.  
\end{theorem}

We first need a Mautner type lemma adapted to the context of "almost invariance" instead of the usual invariance.
\begin{lem}\label{lem:mautner}Let $G$ be a locally compact group, $\theta:G\times G\to \C$ a normalized positive definite kernel such that $\underset{g\in G}{\sup}\Vert g.\theta-\theta\Vert_{cb}<\varepsilon$. Let $x,y\in G$ be such that $$\vert \theta(y^{-1}xy,1)-1\vert <\varepsilon \quad\textrm{and }\quad\vert \theta(y,1)-1\vert <\varepsilon,$$then $$\vert \theta(x,1)-1\vert <2\varepsilon+4\varepsilon^{1/2}.$$
    
\end{lem}

\begin{proof}
First, note that for any $g\in G$, $$\vert\theta(gy,g)-1\vert\leq \vert\theta(gy,g)-\theta(y,1)\vert+\vert \theta(y,1) -1\vert<2\varepsilon.$$
We have \begin{align*}
   \vert \theta(x,1)-1\vert  & \leq  \vert \theta(x,1)-\theta(y^{-1}xy,1) \vert+\vert\theta(y^{-1}xy,1)-1\vert \\
     & <  \vert \theta(x,1) -\theta(y^{-1}x,y^{-1})\vert+\vert \theta(y^{-1}x,y^{-1})-\theta(y^{-1}xy,1) \vert+ \varepsilon \\
     & < 2\varepsilon + \vert \theta(y^{-1}x,y^{-1})-\theta(y^{-1}x,1) \vert +\vert \theta(y^{-1}x,1) -\theta(y^{-1}xy,1) \vert\\
     & < 2\varepsilon + \sqrt{2}\vert \theta(1,y^{-1})-1\vert^{1/2}+ \sqrt{2}\vert \theta(y^{-1}xy,y^{-1}x)-1\vert^{1/2}\\
     & < 2\varepsilon + 2\sqrt{2}\left(2\varepsilon\right)^{1/2}.\qedhere
\end{align*}
\end{proof}

We are now ready to prove Theorem \ref{thm:sp4}

\begin{proof}[Proof of Theorem \ref{thm:sp4}]
Let $\varepsilon>0$ and $(K_0,\delta)$ associated to $\varepsilon$ in property $(T_P)$ for the pair $\left(SL_2(\K)\ltimes S^{2*}(\K^2),S^{2*}(\K^2)\right)$. We may assume $\delta<\varepsilon$. Consider $\iota_1$ and $\iota_2$ the embeddings of $SL_2(\K)\ltimes S^{2*}(\K^2)$ into $G=Sp_4(\K)$, and set $K=\iota_1(K_0)\cup \iota_2(K_0)$.\medskip

Let $\theta$ be a normalised positive definite kernel on $Sp_4(\K)$, that we may assume continuous (by a Remark in \cite[Section 3]{Ozawa+2011+89+104}), such that $$\underset{g\in G}{\sup}\Vert g.\theta-\theta\Vert_{cb}<\delta$$and $$\underset{g^{-1}h\in K}{\sup}\vert \theta(g,h)-1\vert<\delta.$$

Then by relative property $(T_P)$ for $(SL_2(\K)\ltimes S^{2*}(\K^2),S^{2*}(\K^2))$, we get that $$\underset{s\in N_2^+\cup N_2^-}{\sup} \vert \theta(s,1)-1\vert.$$Consider the subgroup $$
H=\left\lbrace\begin{pmatrix}
    a &0&b&0\\
    0&1&0&0\\
    c&0&d&0\\
    0&0&0&1
\end{pmatrix}\mid ad-bc=1\right\rbrace\simeq SL_2(\mathbb{K})$$and its two subgroups $N^+=\{g\in H\vert a=d=1,c=0\}$, $N^-=\{g\in H\vert a=d=1,b=0\}$. Since $N^+\cup N^-\subset N_2^+\cup N_2^-$, for any $s\in N^+\cup N^-$, $\vert \theta(s,1)-1\vert<\varepsilon$.

If $g\in H,s\in  N^+\cup N^-$, we have \begin{align*}
    \vert \theta(gs,g)-1\vert^2 & \leq 2\vert \theta(gs,g)-1\vert  \\
     & \leq 2(\vert  \theta(gs,g)-\theta(s,1)\vert+\vert \theta(s,1)-1\vert)\\
     & <  2(\delta+\varepsilon) < 4\varepsilon.
\end{align*}

But every element $g$ in $H$ can be written as a product of at most $3$ elements of $N^+\cup N^-$ (these corresponds to the transvections in $SL_2(\K)$). Thus, we get that for any $g\in H$, $$\vert \theta(g,1)-1\vert \leq 4\varepsilon^{1/2}+\varepsilon=\varepsilon'.$$

For any $\lambda\in \K$, the matrix $d_\lambda=\Diag\left(\lambda,1,\lambda^{-1},1\right)$ is an element of $H$. For $x\in \K$, consider the matrices $$a(x)=\begin{pmatrix}
    1 & x & 0 & 0\\
    0 & 1 & 0 & 0\\
    0 & 0 & 1 & 0\\
    0 & 0 & -x & 1
\end{pmatrix}\quad \mathrm{ and }\quad a'(x)=\begin{pmatrix}
    1 & 0 & 0 & 0\\
    x & 1 & 0 & 0\\
    0 & 0 & 1 & -x\\
    0 & 0 & 0 & 1
\end{pmatrix}.$$
Let $x\in \K$ fixed. If $\lambda^{-1}\to 0$, we have $d_\lambda^{-1} a(x)d_\lambda\to 1$. In particular, by continuity of $\theta$, there is $\lambda$ such that $\vert \theta(d_\lambda^{-1} a(x)d_\lambda,1)-1\vert < \varepsilon'$. Thus, by Lemma \ref{lem:mautner}, we have $$\vert \theta(a(x),1)-1\vert < 2\varepsilon'+4\sqrt{\varepsilon'}=\varepsilon''.$$

Similarly, if $\lambda\to 0$, we have $d_\lambda^{-1} a'(x)d_\lambda\to 1$ and thus $\vert \theta(a'(x),1)-1\vert < \varepsilon''$.\medskip

Finally, there is some integer $\ell$ such that any element $g\in G$ is a product of at most $\ell$ elements of $a(\K)\cup a'(\K)\cup N_2^+\cup N_2^-$ (see \cite{neuhauser}).

Thus for any $g\in G$, \begin{equation*}\vert \theta(g,1)-1\vert\leq 2\ell \sqrt{\varepsilon''}\end{equation*}which shows that $G$ has $(T_P)$.
\end{proof}

\section{Algebraic groups over local fields}
We now know that $SL_3(\mathbb{K})$ and $Sp_4(\mathbb{K})$ have property $(TTT)$. Following the proof of property $(T)$, we want to show that any   almost $\K$-simple algebraic group of rank at least $2$ has $(TTT)$, where $\K$ is a local field. Before that, we need to show that $(TTT)$ is stable under some operations.\medskip

If $G_1,G_2$ are locally compact group, a quasi-homomorphism is Borel map $\varphi:G_1\to G_2$ which is regular (i.e. the image of a comapct subset of $G_1$ is relatively compact) and such that $\left\{\varphi(gh)^{-1}\varphi(g)\varphi(h)\right\}$ is relatively compact. 

\begin{prop}
    \label{prop:stab_image}Let $G_1,G_2$ be two locally compact groups.
    Let $\varphi:G_1\mapsto G_2$ be a surjective quasi-homomorphism. If $G_1$ has $(TTT)$, then $G_2$ has $(TTT)$.
\end{prop}

\begin{proof}
Let $b$ be Borel wq-cocycle on $G_2$. Then since $\varphi$ is a quasi-homomorphism, $b\circ \varphi$ is a wq-cocycle on $G_1$, hence bounded by $(TTT)$. Since $\varphi$ is surjective, $b$ is bounded.
\end{proof}

\begin{prop}\label{prop:ext}
Let $G$ be a second countable locally compact group, $N\triangleleft G$ a closed normal subgroup. If $N$ and $G/N$ have $(TTT)$, then $G$ has $(TTT)$.    
\end{prop}

\begin{proof}Let $b$ be a wq-cocycle on $G$, and let $$D=\underset{g,h\in G}{\sup} \Vert b(gh)-b(g)-\pi(g)b(h)\Vert<+\infty$$ be its defect. Then $b\vert_N$ is a wq-cocyle on $N$, hence bounded by $C$ by property $(TTT)$. By \cite[Lemma 1.1]{mackey}, there exists a Borel section $\sigma:G/N\to G$ which is regular, meaning that the image of any compact subset of $G/N$ is relatively compact in $G$. Denote $n_g=g {-1}\sigma(gN)$. Set $\Tilde{b}=b\circ\sigma$ and $\Tilde{\pi}=\pi\circ \sigma$. Then $\Tilde{b}$ is a wq-cocycle on $G/N$ associated to $\Tilde{\pi}$. Indeed, if $g,h\in G$, 
\begin{multline*}
\Vert \Tilde{b}(ghN)-\Tilde{b}(gN)-\Tilde{\pi}(gN)\Tilde{b}(hN)\Vert\\
\begin{aligned}
     & = \Vert b(ghn_{gh})-b(gn_g)-\pi(gn_g)b(hn_h)\Vert\\
     &\leq \Vert b(ghn_{gh})-b(gn_ghn_h)\Vert+D\\
     &\leq \Vert b(ghn_{gn})-b(ghn')\Vert+D\\
    & \leq \Vert b(ghn_{gh})-b(gh)-\pi(gh)b(n_{gh)}\Vert+\Vert b(gh)-b(ghn')\Vert+\Vert b(n_{gh})\Vert+D\\
    & \leq \Vert b(ghn')-b(gh)-\pi(gh)b(n')\Vert+\Vert b(n')\Vert + 2D+C\\
    & \leq 3D+2C,
\end{aligned}
\end{multline*}
using that $N$ is a normal subgroup, thus $h^{-1}n_gh\in N$. Since $G/N$ has property $(TTT)$, $\Tilde{b}$ is bounded by $C'$. Thus, for any $g\in G$, $$\Vert b(g)\Vert \leq \Vert b(gn_g)-b(g)-\pi(g)b(n_g)\Vert+ \Vert b(gn_g)\Vert+\Vert b(n_g)\Vert \leq D+C+C'$$so $b$ is bounded on $G$, and thus $G$ has property $(TTT)$.
\end{proof}

In \cite[Thm. 6]{Ozawa+2011+89+104}, Ozawa showed that a lattice in a group with property $(TTT)$ inherits property $(TTT)$. In fact, his proof also shows the following results.
\begin{theorem}{\cite[Thm. 6]{Ozawa+2011+89+104}}\label{lem:cocompactsubgroup}
    Let $H$ be a closed subgroup of $G$ locally compact second countable such that there exists a finite Borel measure on $G/H$ invariant under the action of $G$. If $G$ has property $(TTT)$, then $H$ has property $(TTT)$.
\end{theorem}

We will now turn to algebraic groups. By algebraic group, we will always mean an affine algebraic group realised as an algebraic subgroup of $GL_n$. We will use the notations of \cite[Ch. I]{margulis1991discrete}, where more details can be found.

\begin{lem}\label{lem:simplyconnected}
    Let $\K$ be a local field, $G$ a connected semisimple $\K$-group and $\Tilde{G}$ its simply connected cover (in the algebraic sense). Then $G(\K)$ has $(TTT)$ if and only if $\Tilde{G}(\K)$ has $(TTT)$.
\end{lem}

\begin{proof}
    Let $\pi:\Tilde{G}\to G$ be a central $\K$-isogeny. Then by \cite[Ch. I, Thm. 2.3.4]{margulis1991discrete}, $\pi(\Tilde{G}(\K))$ is a closed normal subgroup of $G(\K)$ such that $G(\K)/\pi(\Tilde{G}(\K))$ is compact (thus has $(TTT)$ as well as a finite Borel measure invariant by $G(\K)$). By Proposition \ref{prop:ext}, $\pi(\Tilde{G}(\K))$ has $(TTT)$ implies $G(\K)$ has $(TTT)$. Conversely, by Theorem \ref{lem:cocompactsubgroup}, if $G(\K)$ has $(TTT)$, so does $\pi(\Tilde{G}(\K))$.\medskip
    
    Furthermore, $\Tilde{G}(\K)/(\ker\pi)(\K)\to \pi(\Tilde{G}(\K))$ is a homeomorphism. Thus, since the subgroup $(\ker\pi)(\K)$ is finite hence has $(TTT)$, by Propositions \ref{prop:stab_image} and \ref{prop:ext}, $\Tilde{G}(\K)$ has $(TTT)$ if and only if $\pi(\Tilde{G}(\K))$ has $(TTT)$.
\end{proof}

In \cite{cornulier}, Yves de Cornulier studied lengths on algebraic groups and showed the following theorem. A semigroup length on $G$ is a map $\ell:G\to \R_+$ which is locally bounded and such that $\forall x,y\in G$, $\ell(xy)\leq \ell(x)+\ell(y)$.
\begin{theorem}{\cite[Thm. 1.4]{cornulier}}
    Let $G$ be an almost $\mathbb{K}$-simple algebraic group over a local field $\mathbb{K}$, then every semigroup length on $G(\mathbb{K})$ is bounded or proper. 
\end{theorem}

To prove Theorem \ref{thm:mainthm}, we will show using $(TTT)$ on $SL_3$ and $Sp_4$ that a certain length is not proper, thus is bounded.

\begin{theorem}\label{thm:mainthmP}
    Let $\K$ be a local field, $G$ a connected almost $\K$-simple $\K$-group with $\rank_\K G \geq 2$. Then $G(\K)$ has property $(T_P)$.
\end{theorem}

\begin{proof}
    By \cite[Ch.I, Prop. 1.6.2]{margulis1991discrete}, $G$ contains an almost $\K$-simple $\K$-subgroup $H$ whose (algebraic) simply connected cover is $SL_3$ or $Sp_4$. Thus, by \cite[Thm.5]{Ozawa+2011+89+104}, Theorem \ref{thm:sp4} and Lemma \ref{lem:simplyconnected}, $H(\K)$ has property $(TTT)$.\medskip

 Let $b$ be a wq-cocycle on $G(\K)$. Let $C=\underset{g,h\in G}{\sup} \left\Vert b(gh)-b(g)-\pi(g)b(h)\right\Vert<+\infty$. Then $b\vert_{H(\K)}$ is a wq-cocycle on $H(\K)$ hence is bounded by $(TTT)$.\medskip

Consider the function $\ell:g\mapsto \Vert b(g)\Vert+C$. We have that $\ell(gh)\leq \ell(g)+\ell(h)$. Furthermore, $\ell$ is locally bounded since by definition $b$ is. Then, by \cite[Thm. 1.4]{cornulier}, $\ell$ is either proper or bounded. But $b$ is bounded on $H(\K)$ which is not relatively compact, thus $b$ is bounded.
\end{proof}

\begin{remark}
    Let $G$ be a connected simple Lie group with finite center of rank at least $2$. Then $G$ is locally isomorphic to the group of $\R$-point of an almost-$\R$-simple algebraic group, thus has $(TTT)$.
\end{remark}

\begin{coro}
    Let $\K$ be a local field, $G$ a connected almost $\K$-simple $\K$-group with $\rank_\K G \geq 2$. Let $\Gamma$ be a lattice in $G(\K)$, then $\Gamma$ has $(TTT)$.
\end{coro}
\begin{proof}
    This is a direct consequence of the theorem and the fact that $(TTT)$ passes to lattices.
\end{proof}

Let $\varphi:G\to G'$ be a quasi-homomorphism. As noticed by Ozawa in \cite{Ozawa+2011+89+104}, if $b$ is a wq-cocycle on $G'$, then $b\circ \varphi$ is a wq-cocycle on $G$. Hence, if $G$ has property $(TTT)$ and there exists $b$ a proper wq-cocycle on $G'$ (i.e. such that $\left\{g\vert \Vert b(g)\Vert\leq n \right\}$ is relatively compact in $G'$ for any $n\in \N^*$), then any quasi-homomorphism $G\to G'$ has a relatively compact image.
\begin{coro}
Let $\Gamma$ be a lattice in an higher rank almost $\K$-simple algebraic group, then any quasi-homomorphism $\Gamma\to G'$ where $G'$ admits a proper wq-cocycle has relatively compact image.
\end{coro}
This applies in particular when $G'$ has Haagerup property, or when $G'$ is hyperbolic. Thus, it gives another proof of \cite[Coro. 4.3]{fujiwara2016quasihomomorphisms}.

\section{Simple Lie groups with infinite center}
In the previous section, we showed that any connected simple with finite center of rank at least $2$ has $(TTT)$. We say that a quasi-homomorphism $\Phi:G\to \R$ is homogeneous if for any $g\in G,n\in \N$, $\Phi(g^n)=n\Phi(g)$. In that case, if $g,h$ commute, then $\Phi(gh)=\Phi(g)+\Phi(h)$. Let $G$ be a connected simple Lie group with infinite center $Z(G)$ and rank at least $2$. Then by \cite[Prop. 6]{Barge1992},  the space of homeogenous quasi-morphism is one dimensional. In particular, a nonzero element of this space is a wq-cocyle (and even a quasi-cocycle) which is unbounded, thus $G$ does not have property $(TTT)$ (and $(TT)$ as well).\medskip

Let $\mathfrak{g}=\mathfrak{k}\oplus \mathfrak{p}$ be a Cartan decomposition of the Lie algebra of $G$ and $\mathfrak{a}$ be a maximal abelian subspace of $\mathfrak{p}$. Let $A,K$ be the analytic subgroups of $G$ with Lie algebras $\mathfrak{a},\mathfrak{k}$ respectively. Then $G=KAK$ as in the finite center case. However, note that $K$ is not compact. Indeed, $Z(G)\subset K$ is an infinite discrete subgroup, but $K/Z(G)$ is compact. 

The following lemma is due to Yves de Cornulier and Mikael de la Salle in an unpublished note. We here reproduce their proof.
\begin{lem}\label{lem:section}
    Let $G$ be a connected simple Lie group with infinite center. There exists a Borel regular section $s:G/Z(G)\to G$ such that $$S=\{s(ghZ(G))(s(gZ(G))s(hZ(G)))^{-1} \mid g,h\in G\}$$ is finite and $s(\exp_{\Ad(G)}(X))=\exp_G(X)$ for any $X\in \mathfrak{a}$.
\end{lem}
\begin{proof}
    Let $\Phi$ be a Barge-Ghys morphism, normalized by $\Phi(Z(G))\subset \Z$. Since $\Phi$ is homegeneous, we can define $s(gZ(G))$ by $g$ if $\Phi(g)\in \left[-\frac{1}{2},\frac{1}{2}\right]$. Since $\Phi$ is a quasimorphism, there is $C>0$ such that $\vert \Phi(gh)-\Phi(g)-\Phi(h)\vert \leq C$. But we have \begin{align*}
        \vert \Phi(s(xy)s(y)^{-1}s(x)^{-1}) \vert & \leq 2C + \vert \Phi(s(xy))\vert + \vert \Phi(s(y)^{-1})\vert+ \vert\Phi(s(x)^{-1})\vert\\
        & \leq 2C+\frac{3}{2}
    \end{align*}bounded independently of $x,y\in G/Z(G)$. Since $\{s(xy)s(y)^{-1}s(x)^{-1}\}\subset Z(G)$, it is finite.
\end{proof}

Note that $s(1)=1$. We want to study the wq-cocycle on $G$, up to bounded functions. Let $H$ be an Hilbert space and $\pi:G\to \mathcal{U}(H)$ be fixed. Let $$Z_w(G,\pi)=\{b:G\to H \mid b\textrm{ wq-cocycle for }\pi\}$$ and $B_w(G,\pi)$ the subspace of bounded Borel functions. We want to understand the space $H_w(G,\pi)=Z_w(G,\pi)/B_w(G,\pi)$.\medskip

Let $i:Z(G)\to G$ denote the inclusion, right composition by $i$ induces a map $$i_*:H_w(G,\pi)\to H_w(Z(G),\pi).$$ Denote $z_g=gs(gZ(G))^{-1}\in Z(G)$.
\begin{prop}\label{prop:infinitecenter}
    The map $i_*$ is injective and $$i_*\left(H_w(G,\pi)\right)=\left\lbrace [b] \mid \underset{g\in G,z\in Z(G)}{\sup}\Vert \pi(g)b(z)-\pi(z_g)b(z) \Vert<+\infty\right\rbrace.$$
\end{prop}
\begin{proof}
    Let $b$ be a wq-cocycle with defect $D$ such that $i_*[b]=0$, then $b\circ i$ is bounded. The map $b\circ s$ is also a wq-cocycle on $G/Z(G)$. Indeed, since $s$ is Borel regular, $b\circ s$ is Borel and locally bounded. Furthermore, if $g,h\in G/Z(G)$, then \begin{align*}
        \Vert b(s(gh))-b(s(g))-\pi(s(g))b(s(h))\Vert & \leq \Vert b(s(gh))-b(s(g)s(h))\Vert+D\\
        & \leq \Vert b((s(g)s(h))^{-1}s(gh))\Vert+2D
    \end{align*}which is bounded in $g,h$ since $S$ is finite. But then since by Theorem \ref{thm:mainthmP}, $b\circ s$ is bounded.\\
    Let $g\in G$, then $g=z_gs(gZ(G))$. Thus $$\Vert b(g)\Vert \leq D+\Vert b(s(gZ(G))\Vert+\Vert b(z_g)\Vert$$so $b$ is bounded and $[b]=0$.\medskip

    Let $\Tilde{b}$ be a wq-cocycle on $Z(G)$ with defect $D$. If there exists a wq-cocycle $b$ on $G$ with defect $D'$ such that $i_*[b]=[\Tilde{b}]$,
    Thus, for any $g\in G,z\in Z(G)$, using that $z,z_g$ commute with $G$, \begin{align*}
        \Vert \pi(g)b(z)-\pi(z_g)b(z)\Vert & \leq \Vert \pi(g)b(z)+b(z_gs(gZ(G)))-b(gz)\\
        & \phantom{{}\leq \Vert} -b(z_gs(gZ(G)))+b(z_g)+\pi(z_g)b(s(gZ(G)))\\
        & \phantom{{}\leq \Vert} +b(gz)-b(z_g)-\pi(z_g)b(zs(gZ(G)))\\
        & \phantom{{}\leq \Vert} +\pi(z_g)b(zs(gZ(G)))-\pi(z_g)b(z)-\pi(z_g)\pi(z)b(s(gZ(G)))\\
        & \phantom{{}\leq \Vert}- \pi(z_g)b(s(gZ(G)))+\pi(z_g)\pi(z)b(s(gZ(G))) \Vert\\
        & \leq 4D'+ 2\Vert b(s(gZ(G)))\Vert.
    \end{align*}But since $b\circ s$ is bounded, we get that $$\underset{g\in G,z\in Z(G)}{\sup}\Vert \pi(g)b(z)-\pi(z_g)b(z)\Vert<+\infty.$$Finally since $b\vert_{Z(G)}-\Tilde{b}$ is bounded by assumption, we get the necessary condition\begin{equation}\label{eq:condition}
        \underset{g\in G,z\in Z(G)}{\sup}\Vert \pi(g)\Tilde{b}(z)-\pi(z_g)\Tilde{b}(z) \Vert<+\infty
    \end{equation}
    Finally, we show that condition \eqref{eq:condition} is sufficient. Let $C$ be the supremum. Define $b(g)=\Tilde{b}(gs(gZ(G))^{-1})=\Tilde{b}(z_g)$ which is Borel and locally bounded. Then $b$ is a wq-cocycle. Indeed, if $g,h\in G$, then \begin{align*}
    \Vert b(gh)-b(g)-\pi(g)b(h)\Vert & = \Vert \Tilde{b}(z_{gh})-\Tilde{b}(z_g)-\pi(g)\Tilde{b}(z_h)\Vert \\
    & \leq \Vert \Tilde{b}(z_{gh})-\Tilde{b}(z_g)-\pi(z_g)\Tilde{b}(z_h)\Vert+\Vert  \pi(g)\Tilde{b}(z)-\pi(z_g)\Tilde{b}(z)\Vert\\
    & \leq \Vert  \Tilde{b}(z_{gh})-\Tilde{b}(z_gz_h)\Vert + D + C\\
    & \leq \Vert  \Tilde{b}((z_gz_h)^{-1}z_{gh})\Vert + 2D+C.
    \end{align*}But $(s(gZ(G))s(hZ(G)))^{-1}s(ghz(G))=(z_gz_h)^{-1}z_{gh}$ so since $S$ is finite, $$ \underset{g,h\in G}{\sup}\Vert  \Tilde{b}((z_gz_h)^{-1}z_{gh})\Vert<+\infty.$$
    Finally, for any $z\in Z(G)$, $b(z)=\Tilde{b}(zs(1)^{-1})=\Tilde{b}(z)$ so that $i_*[b]=[\Tilde{b}]$.
\end{proof}
\begin{remark}In particular, any wq-cocycle on $Z(G)$ associated with $\pi:Z(G)\to \mathcal{U}(H)$ induces a wq-cocycle on $G$, for $\pi':G\to \mathcal{U}(H)$ defined by $$\pi'(g)=\pi(gs(gZ(G))^{-1})=\pi(z_g).$$
Furthermore, any wq-cocycle on $G$ is bounded on $A$, since $A\subset s(G/Z(G))$.
\end{remark}


\addcontentsline{toc}{section}{Bibliography}
\bibliographystyle{alpha}
\bibliography{Ref}
\end{document}